\documentclass[11pt,a4paper,halfparskip*]{scrartcl}

\usepackage[english]{babel}
\usepackage[latin1]{inputenc}
\usepackage[T1]{fontenc}
\usepackage{mathtools}
\usepackage{amsmath}
\usepackage{amssymb}
\usepackage{amscd}
\usepackage{amsthm}
\usepackage{color}      % Colored text
\usepackage{enumitem}  % Improved enumerate style
\usepackage{hyperref}
\usepackage{xspace}
\usepackage{dcolumn}
\usepackage{microtype}

\mathtoolsset{centercolon}

\setitemize{leftmargin=18pt,topsep=5pt,itemsep=2pt,parsep=0pt}
\setenumerate{leftmargin=22pt,itemsep=1pt,parsep=1pt,label={\rm (\alph*)}}
\setitemize[2]{label=$\triangleright$}

\newcommand{\ol}{\overline}
\newcommand{\cL}{\mathcal L}
\newcommand{\F}{\mathbb F}
\newcommand{\Z}{\mathbb Z}
\newcommand{\Frat}{\phi} % Frattini group
\newcommand{\gen}[1]{{\langle{#1}\rangle}}
\newcommand{\GL}{\mathrm{GL}}
\newcommand{\Defn}[1]{\textcolor{blue}{\textit{#1}}}

\newcommand{\abs}[1]{\lvert#1\rvert}
\newcommand{\autGM}{{\Aut_M}(G^*)}

\DeclareMathOperator{\Aut}{Aut}
\DeclareMathOperator{\Stab}{Stab}
\DeclareMathOperator{\Comp}{Comp}

\DeclareMathOperator{\im}{im}
\DeclareMathOperator{\id}{id}

\theoremstyle{plain}
  \newtheorem{theorem}{Theorem}

  \newtheorem{lemma}[theorem]{Lemma}

  \newtheorem*{algo*}{Algorithm}

\theoremstyle{definition}

  \newtheorem{remark}[theorem]{Remark}

\title{The construction of finite solvable groups \\ revisited}
\author{Bettina Eick \and Max Horn}

\begin{document}

\maketitle

\begin{abstract}
We describe a new approach towards the systematic construction of finite
groups up to isomorphism. This approach yields a practical algorithm for
the construction of finite solvable groups up to isomorphism. We report
on a GAP implementation of this method for finite solvable groups and
exhibit some sample applications.
\end{abstract}

%%%%%%%%%%%%%%%%%%%%%%%%%%%%%%%%%%%%%%%%%%%%%%%%%%%%%%%%%%%%%%%%%%%%%%%%%%%%%
\section{Introduction}

The construction of all groups of a given order is an old and fundamental 
topic in finite group theory. Given an order $n$, the aim is to determine 
a list of groups of order $n$ so that every group of order $n$ is isomorphic 
to exactly one group in the list. There are many contributions to this topic
in the literature. In the early history these are based on hand 
calculations; in more recent years algorithms have been developed for this 
purpose. We refer to \cite{BEO02} for a historic overview and a survey of
the available algorithms.

Modern group construction algorithms  distinguish three cases: nilpotent 
groups, solvable groups and non-solvable groups. Nilpotent groups are 
determined as direct products of $p$-groups and $p$-groups can be constructed 
using the $p$-group generation algorithm \cite{OBr90}. Solvable groups can 
be determined by the Frattini extension method \cite{BEi99} or the cyclic 
split extensions methods \cite{BEi00}. Non-solvable groups can be obtained 
via cyclic extensions of perfect groups as in \cite{BEi99} or via the 
method in \cite{Arc98}.

The $p$-group generation algorithm has been used to determine the groups
of order dividing $2^9$ \cite{EOB99} and the construction of the groups of 
order dividing $p^7$ for all primes $p$ is also based on it, see \cite{NOV04} 
and \cite{OVL05}. The algorithm can also be used to determine groups with 
special properties; for example, it is a main tool in the investigation of 
finite $p$-groups by coclass, see \cite{LGM02} for background, and in the
construction of restricted Burnside groups, see \cite{NOb96} and \cite{OVa02}. 
The $p$-group generation algorithm reduces the isomorphism problem to an 
orbit-stabilizer calculation.

The Frattini extension method and the cyclic split extension method have
been used to determine most solvable non-nilpotent groups of order at most 
2000, see \cite{BEO02}. The cyclic split extension method applies to groups
with normal Sylow subgroup and order of the form $p^n \cdot q$ for different 
primes $p$ and $q$ only, while the Frattini extension method applies to all 
non-nilpotent solvable groups. The Frattini extension method uses a random 
isomorphism test for the reduction to isomorphism types. 

The central aim of this paper is to introduce a new approach towards the
systematic construction of groups up to isomorphism. This new approach is
particularly useful for finite solvable groups and we developed it in
detail and implemented it in GAP \cite{GAP} for this case. For finite 
$p$-groups, our new approach coincides with the approach of $p$-group 
generation. In particular, the new approach reduces the solution of the 
isomorphism problem to an orbit-stabilizer calculation.

We discuss some sample applications of our new approach. We determined 
(again) all solvable non-nilpotent groups of order at most 2000 and thus 
check the results available in the Small Groups Library \cite{BEO00}
and we constructed (for the first time) the groups of order 
$2304 = 3^2 \cdot 2^8$. We believe that our new approach could also be
useful in the experimental investigation of coclass theory for finite
solvable groups as suggested in \cite{HPl93} or in the construction of
other finite solvable groups with special properties.

%%%%%%%%%%%%%%%%%%%%%%%%%%%%%%%%%%%%%%%%%%%%%%%%%%%%%%%%%%%%%%%%%%%%%%%%%%%%%
\section{The general approach of the algorithm}

In this section we exhibit a top-level introduction towards our new approach.
The central idea is to use induction along a certain series: the so-called 
\Defn{$F$-central series}. We first introduce and investigate this series. 
Throughout this section, let $G$ be a finite group.

Recall that the \Defn{Fitting subgroup} $F(G)$ is the maximal nilpotent 
normal subgroup of $G$. Define $\nu_0(G) = F(G)$ and let $\nu_{i+1}(G)$ 
be the smallest normal subgroup of $F(G)$ so that $\nu_i(G)/\nu_{i+1}(G)$ 
is a direct product of elementary abelian groups which is centralized by 
$F(G)$. Then we define the \Defn{$F$-central series} of $G$ as
\[ G \geq \nu_0(G) \geq \nu_1(G) \geq \ldots. \]

The following lemma provides an alternative characterization of the terms
of the $F$-central series. We omit its straightforward proof. For an 
integer $n$ with prime factorisation $n = p_1^{e_1} \cdots p_r^{e_r}$ 
for different primes $p_1, \ldots, p_r$ and exponents $e_1, \ldots, e_r 
\neq 0$, we call $p_1 \cdots p_r$ the \Defn{core} of $n$.

\begin{lemma} \label{altchar}
Let $G$ be a finite group and let $k$ be the core of $|F(G)|$. Then
\[ \nu_{i+1}(G) = [F(G), \nu_i(G)] \nu_i(G)^k 
   \;\; \text{ for each }\;\; i \geq 0.\]
\end{lemma}

If $F(G)$ is a finite $p$-group, then the series $F(G) = \nu_0(G) \geq 
\nu_1(G) \geq \ldots$ coincides with the lower exponent-$p$ central series 
of $F(G)$. In general, the group $F(G)$ is nilpotent and the series $F(G) = 
\nu_0(G) \geq \nu_1(G) \geq \ldots$ is a central series of $F(G)$. Thus
there exists an integer $c$ which $\nu_c(G) = \{1\}$. We call the smallest 
such integer the \Defn{$F$-class} of $G$. Further, if $G$ has $F$-class at
least $1$, then the order of the quotient $\nu_0(G)/\nu_1(G)$ is called the 
\Defn{$F$-rank} of $G$. The next two lemmas collect some elementary facts 
about the $F$-central series. Let $\Frat(G)$ denote the Frattini subgroup 
of $G$.

\begin{lemma} \label{nu-properties1}
Let $G$ be a finite group. 
\begin{enumerate}
 \item $\nu_1(G) = \Frat(F(G)) \leq \Frat(G)$.
 \item $\nu_i(G)$ is a characteristic subgroup of $G$ for each $i \geq 0$.
 \item $\nu_j(G/\nu_i(G)) = \nu_j(G)/\nu_i(G)$ for each $i \geq j \geq 0$ and $i\geq 1$.
\end{enumerate}
\end{lemma}

\begin{proof}
(a) and (b) are elementary and we consider (c) only. 
Suppose $j = 0$, then we need to show $F(G/\nu_i(G))=F(G)/\nu_i(G)$.
Let $L\leq G$ be defined by $L/\nu_i(G)=F(G/\nu_i(G))$.
As $F(G)/\nu_i(G)$ is a nilpotent normal subgroup of $G/\nu_i(G)$,
it follows that $F(G)\leq L$.
Moreover, since  $i\geq 1$ it follows together with (a) that $\nu_i(G) \leq \Frat(G)\leq F(G)\leq L$.
By construction $L/\nu_i(G)$ is nilpotent, hence $L/\phi(G)$ is also nilpotent.
Additionally, $L$ is finite and normal in $G$, thus $L$ is nilpotent by Gasch\"utz' theorem (see \cite[5.2.15]{Rob82}).
In summary we conclude $L\leq F(G)$, and in fact $L=F(G)$, proving the claim.
For $j > 0$ the result follows by induction using Lemma \ref{altchar}.
\end{proof}

\begin{lemma} \label{nu-properties2}
Let $G$ be a finite group. 
\begin{enumerate}
 \item $G$ is solvable if and only if $G/\nu_0(G)$ is solvable.
 \item $G$ is nilpotent if and only if $G/\nu_0(G)$ is trivial.
 \item If $G \neq \{1\}$ is solvable, then $\nu_0(G) \neq \{1\}$ and hence $G$ 
       has $F$-class at least $1$.
\end{enumerate}
\end{lemma}

\begin{proof}
All three items follow directly from the fact that $\nu_0(G) = F(G)$ is 
the maximal nilpotent normal subgroup of $G$. 
\end{proof}

\begin{remark}
\label{frank0}
Let $G$ be a finite group of $F$-class $0$. Then $G$ does not have any
non-trivial solvable normal subgroup. The structure of such groups and 
their construction up to isomorphism is well understood, see 
\cite[p. 89f]{Rob82} for details. 
\end{remark}

Let $G$ be a finite group of $F$-class $c$. A group $H$ is a 
\Defn{descendant} of $G$ (and $G$ is the \Defn{ancestor} of $H$) 
if $H$ has $F$-class $c+1$ and $H / \nu_c(H) \cong G$.  Lemma 
\ref{nu-properties1} asserts that $G$ and $H$ have the same $F$-rank.

Our approach to construct finite groups uses the $F$-class and 
the $F$-rank as primary invariants. Its two central ingredients are the 
following algorithms.

\begin{enumerate}
\item {\bf Algorithm I} \\
Given a positive integer $\ell$, determine up to isomorphism all 
finite groups of $F$-class 1 and $F$-rank $\ell$.
\item {\bf Algorithm II} \\
Given a finite group $G$, determine up to isomorphism all 
descendants of $G$.
\end{enumerate}

Algorithm I and an iterated application of Algorithm II thus yield a
method to determine up to isomorphism all finite groups of $F$-class
$c > 1$ and $F$-rank $\ell$. If Algorithm I is restricted to determine
solvable groups only, then this approach yields an algorithm to determine
up to isomorphism all finite solvable groups of $F$-class $c$ and 
$F$-rank $\ell$, as the following remark asserts.

\begin{remark}
Let $H$ be a descendant of $G$. Then $G$ is solvable if and only if
$H$ is solvable.
\end{remark}

In the following Sections \ref{fclass1} and \ref{fdesc} we discuss 
methods for Algorithm I and II in more detail. In both cases we are 
particularly interested in the construction of solvable groups. We
then briefly describe the GAP implementation of our method in Section
\ref{impl} and we exhibit applications in Section \ref{applics}.

%%%%%%%%%%%%%%%%%%%%%%%%%%%%%%%%%%%%%%%%%%%%%%%%%%%%%%%%%%%%%%%%%%%%%%%%%%%%%
\section[Finite groups of F-class 1]{Finite groups of $F$-class 1}
\label{fclass1}

In this section we describe an effective method to determine up to 
isomorphism the finite groups $G$ of $F$-class 1 and given $F$-rank $\ell$. 
Every such group $G$ has a Fitting subgroup $F(G)$ which is a direct product 
of elementary abelian groups and has order $\ell$. The following lemma analyzes 
the structure of such groups $G$ further.

\begin{lemma}
\label{fitting}
Let $G$ be a finite group of $F$-class $1$ and $F$-rank $\ell$. Let
$H$ be the maximal solvable normal subgroup of $G$. 
\begin{enumerate}
\item
$F(G)$ is a direct product of elementary abelian groups and has order $\ell$.
\item
$H$ is solvable of $F$-class $1$ and $F$-rank $\ell$ with $F(G) = F(H)$.
\item
$F(H)$ is self-centralizing in $H$, i.e. $C_H(F(H))=F(H)$.
\end{enumerate}
\end{lemma}

\begin{proof}
(a) As $G$ is $F$-class 1, it follows that $\nu_1(G) = \{1\}$ and thus 
$F(G) = \nu_0(G)$ is a direct product of elementary abelian groups and has 
order $\ell$. \\
(b) As $H$ is characteristic in $G$, it follows that $F(H)$ is normal in
$G$. As $F(H)$ is nilpotent and $F(G)$ is the maximal nilpotent normal
of $G$, this yields $F(H) \subseteq F(G)$. Conversely, $F(G) \leq H$ and
thus $F(G)$ is a nilpotent normal subgroup of $H$. Hence $F(G) \leq F(H)$.\\
(c) This follows from \cite[5.4.4]{Rob82}.
\end{proof}

Hence a finite group $G$ of $F$-class $1$ and $F$-rank $\ell$ can be 
constructed from a finite solvable normal subgroup $H$ of $F$-class $1$ 
and $F$-rank $\ell$ and a quotient $G/H$ of $F$-class $0$. We discuss the
effective construction of the solvable groups of $F$-class $1$ in more 
detail in the following.

\subsection{The solvable case}

Let $G$ be a finite solvable group of $F$-rank $1$ and $F$-class $\ell$.
Then 
Lemma \ref{fitting}(c) asserts $G$ can be written as an extension of $F(G)$
by a solvable subgroup $U \leq \Aut(F(G))$. The isomorphism type of the 
group $F(G)$ is fully determined by $\ell$. We discuss the construction of 
the relevant subgroups $U$ of $\Aut(F(G))$ in the following Section 
\ref{relsubs} and the determination of the corresponding extensions of 
$F(G)$ by $U$ in Section \ref{relexts}. In Section \ref{algoI} we summarize 
the resulting algorithm to construct finite solvable groups of $F$-class 1.

\subsubsection{The relevant subgroups}
\label{relsubs}

Let $A$ be the direct product of elementary abelian groups with $\abs{A}=\ell$. 
Our aim is to determine up to conjugacy those solvable subgroups $U$ of 
$\Aut(A)$ so that there exists an extension $G$ of $A$ by $U$ with $F(G) 
\cong A$. We first investigate these subgroups in more detail.

A subgroup $N$ of $\Aut(A)$ centralizes a series through $A$ if there 
exists an $N$-invariant series $A = A_1 > A_2 > \ldots > A_l > A_{l+1}
= \{1\}$ so that $N$ induces the identity on every quotient $A_i/A_{i+1}$.
We say that a group $U \leq \Aut(A)$ is \Defn{$F$-relevant} if none of 
its non-trivial normal subgroups centralizes a series through $A$.

\begin{lemma} \label{relevant}
Let $A$ be the direct product of elementary abelian groups with $\abs{A}=\ell$
and let $U \leq \Aut(A)$. Then the following are equivalent:
\begin{enumerate}
\item \label{item:not-s-cent-a}
$U$ is $F$-relevant.
\item \label{item:not-s-cent-b}
Every extension $G$ of $A$ by $U$ satisfies $F(G) \cong A$.
\item \label{item:not-s-cent-c}
There exists an extension $G$ of $A$ by $U$ with $F(G) \cong A$.
\end{enumerate}
\end{lemma}

\begin{proof}
\begin{description}[leftmargin=0pt,fullwidth]
\item[\ref{item:not-s-cent-a} $\Rightarrow$ \ref{item:not-s-cent-b}:]
Suppose that $U$ is $F$-relevant and let $G$ be an arbitrary extension
of $A$ by $U$. We consider $A$ as normal subgroup of $G$. Then $A \leq F(G)$,
as $A$ is abelian and thus nilpotent. Further, as $G/A$ corresponds to $U$, 
it follows that $F(G)/A$ corresponds to a normal subgroup $N$ of $U$. As 
$F(G)$ is nilpotent, $N$ centralizes a series through $A$. As 
$U$ is $F$-relevant, $N$ is trivial and hence $A = F(G)$ follows.
\item[\ref{item:not-s-cent-b} $\Rightarrow$ \ref{item:not-s-cent-c}:]
Trivial, since the split extension exists.
\item[\ref{item:not-s-cent-c} $\Rightarrow$ \ref{item:not-s-cent-a}:]
Let $N$ be a normal subgroup of $U$ which centralizes a series through $A$. 
Then $N$ is nilpotent, as it is also a subgroup of $\Aut(A)$.
Let $H$ be the normal subgroup of $G$ with $A \leq H$ and $H/A$ 
corresponding to $N$. Then $H$ is a nilpotent normal subgroup of $G$ and 
$H \leq F(G) = A$ follows. Thus $H=A$ and $N = \{1\}$. Hence $U$ is 
$F$-relevant.
\qedhere
\end{description}
\end{proof}

Lemma \ref{relevant} translates our aim in this subsection to a determination 
up to conjugacy of all solvable $F$-relevant subgroups of $\Aut(A)$. As a 
first step towards this, we give an alternative description for the $F$-relevant 
subgroups of $\Aut(A)$. Let $\ell = p_1^{d_1} \cdots p_r^{d_r}$ be the prime 
factorization of $\ell$. We identify $A$ with $A_1 \times \cdots \times A_r$, 
where $A_i$ is elementary abelian of order $p_i^{d_i}$ for $1 \leq i \leq r$, 
and thus obtain
\[ \Aut(A) = \GL(d_1, p_1) \times \cdots \times \GL(d_r,p_r). \]
For a subgroup $U$ of $\Aut(A)$ and $1 \leq i \leq r$ we write $\sigma_i(U) = 
U \cap \GL(d_i,p_i)$ and we denote with $\pi_i(U)$ the projection of $U$ into 
$\GL(d_i, p_i)$. Further, let $P(U) = \pi_1(U) \times \cdots \times \pi_r(U)$ 
and $S(U) = \sigma_1(U) \times \cdots \times \sigma_r(U)$. Then $U$ is a
subdirect product in $P(U)$ with kernel $S(U)$. Figure \ref{picsubdir} 
illustrates these subgroups for the case $r = 2$.

\begin{figure}[htb]
\begin{center}
\setlength{\unitlength}{0.008in}
\begin{picture}(385,332)(150,335)
\thicklines
\put(200,500){\circle*{5}}
\put(220,480){\circle*{5}}
\put(300,400){\circle*{5}}
\put(360,340){\circle*{5}}
\put(420,400){\circle*{5}}
\put(500,480){\circle*{5}}
\put(520,500){\circle*{5}}
\put(360,460){\circle*{5}}
\put(360,540){\circle*{5}}
\put(360,660){\circle*{5}}
\put(360,620){\circle*{5}}

\put(360,660){\line(-1,-1){160}}
\put(200,500){\line( 1,-1){160}}
\put(360,340){\line( 1, 1){160}}
\put(520,500){\line(-1, 1){160}}
\put(340,640){\line( 1,-1){160}}
\put(380,640){\line(-1,-1){160}}
\put(460,560){\line(-1,-1){160}}
\put(260,560){\line( 1,-1){160}}
\put(360,620){\line( 0,-1){160}}

\put(370,655){\makebox(0,0)[lb]{$\Aut(A)$}}
\put(365,610){\makebox(0,0)[lb]{$P(U)$}}
\put(370,535){\makebox(0,0)[lb]{$U$}}
\put(370,450){\makebox(0,0)[lb]{$S(U)$}}
\put(370,335){\makebox(0,0)[lb]{$\{1\}$}}

\put(170,460){\makebox(0,0)[lb]{$\pi_1(U)$}}
\put(250,380){\makebox(0,0)[lb]{$\sigma_1(U)$}}
\put(420,380){\makebox(0,0)[lb]{$\sigma_2(U)$}}
\put(500,460){\makebox(0,0)[lb]{$\pi_2(U)$}}

\put(110,495){\makebox(0,0)[lb]{$\GL(d_1,p_1)$}}
\put(530,495){\makebox(0,0)[lb]{$\GL(d_2,p_2)$}}

\end{picture}
\caption{The subdirect product for $r = 2$}
\end{center}
\medskip
\label{picsubdir}
\end{figure}
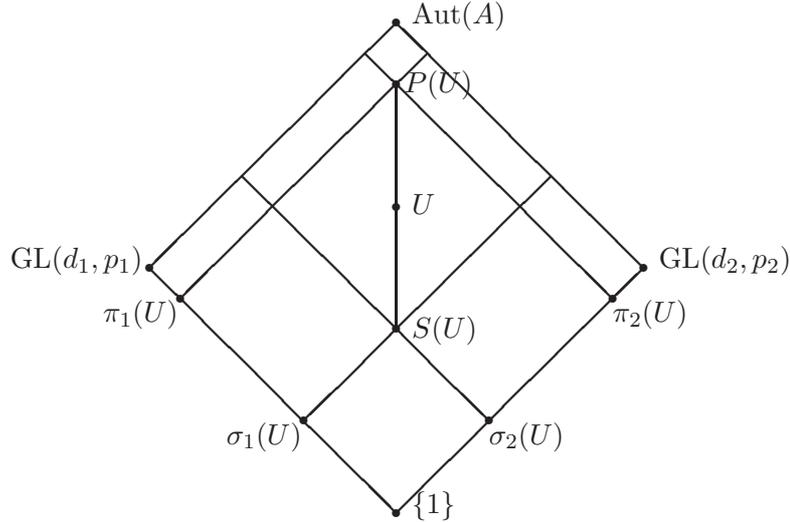

Next, we define $O(U) = O_{p_1}(\sigma_1(U)) \times \cdots \times 
O_{p_r}(\sigma_r(U))$, where $O_p(H)$ is the maximal normal $p$-subgroup 
of the group $H$. Thus we obtain the following series of subgroups
\[ \{1\} \leq O(U) \leq S(U) \leq U \leq P(U) \leq \Aut(A).\]

We note that $\sigma_i(U)$ is normal in $\pi_i(U)$ for $1 \leq i \leq r$
and thus $S(U)$ is normal in $P(U)$. Further, $O(U)$ is characteristic in
$S(U)$ and thus normal in $P(U)$. The following lemma analyzes the situation
further.

\begin{lemma} \label{charrel}
Let $A$ be the direct product of elementary abelian groups with
$\abs{A}=\ell$ and let $U \leq \Aut(A)$. 
\begin{enumerate}
\item 
$O(U)$ is the maximal normal subgroup of $U$ that centralizes a
series through $A$.
\item 
$U$ is $F$-relevant if and only if $O(U) = \{1\}$.
\end{enumerate}
\end{lemma}

\begin{proof}
(b) follows directly from (a) and it remains to prove (a).

Consider the prime factorization $\ell = p_1^{d_1} \cdots p_r^{d_r}$.
Suppose $i \in \{1, \ldots, r\}$ and set $U_i=\sigma_i(U)$.
Then $U_i$ is normal in $U$ and $O_{p_i}(U_i)$ 
is characteristic in $U_i$. Thus $O_{p_i}(U_i)$ is normal in $U$. Further,
$O_{p_i}(U_i)$ is a $p_i$-subgroup of $\GL(d_i,p_i)$ and thus centralizes a 
series through $A_i$. This implies that $O(U)$ is a normal subgroup of $U$
centralizing a series through $A$. 

Conversely, suppose $N$ is a normal subgroup of $U$ centralizing a series
through $A$. For $i \in \{1, \ldots, r\}$ let $P_i$ denote the projection 
of $N$ into $\GL(d_i, p_i)$. Then $P_i$ centralizes a series through $A_i$.
Thus $P_i$ is a $p_i$-group. As all primes $p_1, \ldots, p_r$ are different, 
it follows that $N = P_1 \times \cdots \times P_r$ and $P_i \leq U_i$. Hence 
$P_i \leq O_{p_i}(U_i)$ and $N \leq O(U)$. This yields that $O(U)$ is the 
maximal normal subgroup of $U$ that centralizes a series through $A$.
\end{proof}

Hence our aim translates now to the determination up to conjugacy of all
solvable subgroups $U$ of $\Aut(A)$ with $O(U) = \{1\}$; additionally, we
also determine their normalizers $N(U) = N_{\Aut(A)}(U)$. If $U$ is solvable, 
then $P(U)$ is solvable. We use the following approach.

{\bf Algorithm} {\it RelevantSolvableSubgroups( $\ell$ )} \vspace{-0.3cm}
\begin{itemize}
\item 
Factorize $\ell = p_1^{d_1} \cdots p_r^{d_r}$.
\item
For $1 \leq i \leq r$ determine up to conjugacy all solvable subgroups 
$P_i$ of $\GL(d_i,p_i)$ together with their normalizers 
$R_i = N_{\GL(d_i,p_i)}(P_i)$. \\
(See \cite[Sec. 10.4]{HEO05} for algorithms for this purpose.)
\item
For each combination $P = P_1 \times \cdots \times P_r$ with normalizer
$R = R_1 \times \cdots \times R_r$ determine up to conjugacy all subdirect
products $U$ in $P$ together with their normalizers $N_R(U)$. \\
(See \cite{Eic96} for an algorithm for this purpose.)
\item 
Discard those subdirect products $U$ with $O(U) \neq \{1\}$ and return
the remaining list of subgroups $U$ together with their normalizers 
$N_R(U)$.
\end{itemize}

In the special case $r = 1$ we only need to determine up to conjugacy those 
solvable subgroups $U$ of $\GL(d_1,p_1)$ with $O_{p_1}(U) = \{1\}$. 

\begin{theorem}
Algorithm {\it RelevantSolvableSubgroups( $\ell$ )} determines up to 
conjugacy all solvable $F$-relevant subgroups of $\Aut(A)$ together 
with their normalizers in $\Aut(A)$ where $A$ is the direct product of
elementary abelian groups with $\abs{A}=\ell$.
\end{theorem}

\begin{proof}
We continue to use the notation from the algorithm. 

First we observe that $N_R(U) = N_{\Aut(A)}(U)$. Let $g \in 
N_{\Aut(A)}(U)$. As $\Aut(A)$ is a direct product, it follows that
$g$ also normalizes each $\pi_i(U)$ and thus $P(U)$. Hence $g \in R$
and thus $g \in N_R(U)$. Thus $N_{\Aut(A)}(U) \subseteq N_R(U)$. As 
$N_R(U) \subseteq N_{\Aut(A)}(U)$ is obvious, the result follows.

Completeness: Let $U$ be a solvable subgroup of $\Aut(A)$ with $O(U) = 
\{1\}$. We show that $U$ is conjugate to a subgroup in the returned
list. First $\pi_i(U)$ is conjugate in $\GL(d_i, p_i)$ to a subgroup 
$P_i$ as determined in the algorithm. Hence we may suppose without loss 
of generality that $\pi_i(U) = P_i$ for $1\leq i\leq r$. Thus $P(U) = P$ 
with normalizer $R$ is determined by the algorithm. Then $U$ is a 
subdirect product in $P$ and hence is conjugate to a subgroup determined
by the algorithm.

Non-redundancy: Suppose that two of the determined groups $U$ and $V$ are
conjugate in $\Aut(A)$. Then $P(U)$ is conjugate to $P(V)$ in $\Aut(A)$
and by construction we obtain that $P(U) = P(V) = P$. Thus $U$ and $V$ are
two $\Aut(A)$-conjugate subgroups in $P$. As above in the normalizer
discussion, this implies that $U$ and $V$ are conjugate in $R = 
N_{\Aut(A)}(P)$. Thus $U = V$ by construction.
\end{proof}
 
\subsubsection{The relevant extensions}
\label{relexts}

Let $\ell$ be a positive integer, let $A$ be the direct product of elementary
abelian groups with $\abs{A}=\ell$ and let $U$ be an $F$-relevant subgroup of 
$\Aut(A)$ with normalizer $N(U) = N_{\Aut(A)}(U)$. Our aim is to determine 
up to isomorphism all extensions of $A$ by $U$. 

As described in \cite[p. 315ff]{Rob82}, each element $\delta$ of the group
of 2-cocycles $Z^2(U,A)$ defines an extension $G_\delta$ of $A$ by $U$ and
each extension of $A$ by $U$ is isomorphic to a group of the form $G_\delta$.
Hence the group $Z^2(U,A)$ allows the construction of a complete (but possibly 
redundant) set of isomorphism types of extensions of $A$ by $U$. It remains 
to solve the isomorphism problem.

The subgroup of 2-coboundaries $B^2(U,A)$ in $Z^2(U,A)$ reduces this 
isomorphism problem. If $\delta$ and $\lambda$ are two elements of $Z^2(U,A)$
which belong to the same coset of $B^2(U,A)$, then their corresponding 
extensions $G_\delta$ and $G_\lambda$ are isomorphic, see \cite[p. 316]{Rob82}. 

Another reduction is induced by the normalizer $N(U)$. This acts on 
$Z^2(U,A)$ via $(g(\delta))(u,v) := g( \delta( u^g, v^g ))$ for $\delta 
\in Z^2(U,A)$, $g \in N(U)$ and $u,v \in U$. If $\lambda$ and $\delta$ 
are two elements of $Z^2(U,A)$ which belong to the same orbit under 
$N(U)$, then their corresponding extensions $G_\delta$ and $G_\lambda$ 
are isomorphic, see \cite{Rob81}.  

For $\lambda \in Z^2(U,A)$ we write $[\lambda] := \lambda + B^2(U,A) \in 
H^2(U,A) = Z^2(U,A) / B^2(U,A)$. 
The action of $N(U)$ on $Z^2(U,A)$ leaves $B^2(U,A)$ setwise invariant and 
hence induces an action of $N(U)$ on $H^2(U,A)$. We denote 
$g([\lambda]):=[g(\lambda)]$ for $g\in N(U)$ and $\lambda\in\Z^2(U,A)$. 
Combining the reductions induced by $B^2(U,A)$ and this action,
we obtain that two elements $[\delta]$ and 
$[\lambda]$ in the same $N(U)$-orbit of $H^2(U,A)$ yield isomorphic 
extensions $G_\delta \cong G_\lambda$.

In the special case of $F$-relevant subgroups of $\Aut(A)$, we show in 
the following theorem that the action of $N(U)$ on $H^2(U,A)$ yields a 
full solution of the isomorphism problem. 

\begin{theorem}
\label{isomFC1}
Let $U$ be an $F$-relevant subgroup of $\Aut(A)$ and let $\delta, \lambda 
\in Z^2(U,A)$. Then $G_\delta \cong G_\lambda$ if and only if there 
exists an element $g \in N(U)$ with $g([\delta]) = [\lambda]$.
\end{theorem}

\begin{proof}
It suffices to show that an isomorphism $\iota : G_\delta \to G_\lambda$
implies that $[\delta]$ and $[\lambda]$ are in the same $N(U)$-orbit.

We consider $A$ as additive group and write the extension $G_\delta$ as 
set $\{(u,a) \mid u \in U, a \in A \}$ with multiplication $(u,a)(v,b) 
= (uv, v(a)+b+\delta(u,v))$. Let $A_\delta = \{ (1, a) \mid a \in A \}
\leq G_\delta$. Then $\iota$ maps $A_\delta$ onto $A_\lambda$ and hence
there exists an automorphism $\beta \in \Aut(A)$ with $\iota((1,a)) = 
(1, \beta(a))$. Further, $\iota$ induces an automorphism $\alpha$ of 
the quotient $G_\delta/A_\delta \cong U \cong G_\lambda / A_\lambda$
and thus $\iota$ has the form 
\[ \iota : G_\delta \to G_\lambda 
         : (u,a) \mapsto (\alpha(u), \beta(a) + t(a,u))\]
for some map $t : A\times U \to A$ satisfying $t(a,1)=0$.
Evaluating the equation $\iota((u,a)) = \iota((u,0)) 
\iota((1,a)) = \iota((u,0)) (1, \beta(a))$ shows that $t(a,u) = 
t(0,u)$ for all $u \in U$ and $a \in A$.
We thus write $t(u) = t(0,u)$ to shorten notation.

Evaluating $\iota((1,a))\iota(u,0) = \iota((u,u(a)))$ shows that 
$\alpha(u)(a) = (\beta u \beta^{-1})(a)$ for all $u \in U$ and $a \in A$.
Hence $\alpha(u) = \beta u \beta^{-1}$ for all $u \in U$. Thus $\beta
\in N(U)$ and $\alpha$ is induced by conjugation with $\beta^{-1}$. 

Let $\mu : U \to A : u \mapsto t(\alpha^{-1}(u))$. Then $\mu(1) = t(1) = 0$
and hence $\mu \in C^1(U,A)$. Let $\sigma \in B^2(U,A)$ be the coboundary 
induced by $\mu$, i.e. $\sigma(u,v):=\mu(uv)-v(\mu(u))-\mu(v)$.
Evaluating $\iota((u,0)(v,0)) = \iota(u,0) \iota(v,0)$ yields that
$\beta(\delta(u,v)) + \sigma(\alpha(u), \alpha(v))
  = \lambda(\alpha(u), \alpha(v))$ for every $u, v \in U$. 

Let $g = \beta$. Then $\alpha^{-1}(u) = u^{g}$. Hence $g(\delta(u^g,v^g)) 
+ \sigma(u,v)= \lambda(u,v)$ for every $u, v \in U$. Thus $g([\delta]) 
= [\lambda]$ as desired.
\end{proof}

We note (without proof) that dual to Theorem~\ref{isomFC1} one can also 
determine the automorphism groups of the considered extensions. 

\begin{theorem}
\label{autFC1}
Let $U$ be an $F$-relevant subgroup of $\Aut(A)$ and let $\delta \in 
Z^2(U,A)$. Then 
\[ \eta: \Aut(G_\delta) \to \Stab_{N(U)}([\delta])
       : \alpha \mapsto \alpha|_A\]
is an epimorphism with $\ker(\eta) \cong Z^1(U,A)$.
\end{theorem}

Let $A$ be the direct product of elementary abelian groups with
$\abs{A}=\ell$ and let $U \leq \Aut(A)$ be an $F$-relevant solvable subgroup.
The following algorithm determines all extensions $G$ of $A$ by $U$ 
up to isomorphism together with $\Aut(G)$. 

{\bf Algorithm} {\it RelevantExtensions( $U$, $A$ )} \vspace{-0.3cm}
\begin{itemize}
\item
Compute $H^2(U,A)$ and $Z^1(U,A)$.
\item
Determine the orbits and stabilizers of $N(U)$ on $H^2(U,A)$. 
\item
For each orbit representative $[\delta] \in H^2(U,A)$, compute a
corresponding extension $G_\delta$.
\item
Determine $\Aut(G_\delta)$ from the stabilizer of $[\delta]$ in $N(U)$
and $Z^1(U,A)$.
\end{itemize}

Let $A = A_1 \times \cdots \times A_r$ with $A_i$ elementary abelian of
order $p_i^{d_i}$. Then $H^2(U,A) = H^2(U,A_1) \times \cdots \times 
H^2(U,A_r)$ and, similarly, $Z^1(U,A) = Z^1(U,A_1) \times \cdots \times
Z^1(U,A_r)$. The cohomology groups $H^2(U,A_i)$ and $Z^1(U, A_i)$ can
be computed with the algorithms in \cite{HEO05} for polycyclic groups.

\subsubsection{Algorithm I for solvable groups}
\label{algoI}

We now combine the results of this section to an algorithm to determine
all finite solvable groups of $F$-class 1 and $F$-rank $\ell$. 

{\bf Algorithm} {\it SolvableGroupsOfFClass1( $\ell$ )} \vspace{-0.3cm}
\begin{itemize}
\item 
Let $A$ be the direct product of elementary abelian groups with $\abs{A}=\ell$.
\item
Determine up to conjugacy all $F$-relevant finite solvable subgroups $U$ 
in $\Aut(A)$ together with their normalizers $N(U)$; see Section \ref{relsubs}.
\item
For each such $U$, determine all extensions $G$ of $A$ by $U$ up to 
isomorphism together with their automorphism groups $\Aut(G)$; see Section 
\ref{relexts}.
\end{itemize}

\begin{theorem}
Algorithm {\it SolvableGroupsOfFClass1} determines a complete and irredundant 
set of finite solvable groups $G$ with $F$-class 1 and $F$-rank $\ell$.
\end{theorem}

\begin{proof}
First we observe that every group $G$ in the output of Algorithm I is of the
required type. As $G$ is an extension of $A$ by $U$, it follows that it is
finite solvable. Lemma \ref{relevant} asserts that $G$ has $F$-class 1 and 
$F(G) \cong A$. Hence $G$ has $F$-rank $\ell$.

Completeness:
Let $G$ be a finite solvable group of $F$-class 1 with $F$-rank $\ell$.
Then $G$ is an extension of $A$ by some $V \leq \Aut(A)$ by Lemma
\ref{fitting}. The group $V$ is $F$-relevant by Lemma \ref{relevant}.
Hence $V$ is conjugate $\Aut(A)$ to some subgroup $U$ determined in 
Step 1 of Algorithm I. We can consider consider $G$ as extension of
$A$ by $U$. Using Theorem \ref{isomFC1} it follows that $G$ is 
isomorphic to a group in the output of Algorithm I.

Non-redundancy:
Suppose $G_1$ and $G_2$ in the output of Algorithm I are isomorphic.
Let $G_i$ be an extension of $A$ by $U_i$ for $i = 1,2$. Then $U_1$
and $U_2$ are conjugate in $\Aut(A)$. We write $U = U_1$ and assume 
that $G_2$ is also an extension of $A$ by $U$, as the conjugation from 
$U$ to $U_2$ induces an isomorphism of extensions. Using Theorem 
\ref{isomFC1} we now obtain that $G_1 = G_2$.
\end{proof}

\subsubsection{An example}

We determine the groups of order $96 = 2^5 \cdot 3$ with $F$-class $1$. 
Each such group is solvable and has $F$-rank $\ell \in \{2^n \cdot 3^m 
\mid 0 \leq n \leq 5, 0 \leq m \leq 1\}$. 

{\bf Case 1:} $\ell = 2^n$ with $0 \leq n \leq 5$. Let $A$ be the
elementary abelian group of order $\ell$. Then $\Aut(A) = \GL(n, 2)$.
The relevant subgroups $U$ of $\Aut(A)$ are the solvable subgroups
of $\GL(n,2)$ of order $2^{5-n} \cdot 3$ with $O_2(U) = \{1\}$. The
following table lists the possible isomorphism types for $U$ and
their value $n$, the number of conjugacy classes in $\Aut(A)$ with 
this isomorphism type of subgroup $U$ and the relevant extensions
with this isomorphism type of group $U$; the extensions are described
by their number in the small groups library.

\begin{center}
\begin{tabular}{c|c|c|c}
$n$ & $U$ & \# of classes & extensions \\
\hline
4 & $S_3$ & 3 & (96, 194), (96, 195), (96, 226), (96, 227) \\
5 & $C_3$ & 2 & (96, 228), (96, 229)
\end{tabular}
\end{center}

{\bf Case 2:} $\ell = 2^n \cdot 3$ with $0 \leq n \leq 5$. Let
$A$ be the group $C_2^n \cdot C_3$. Then $\Aut(A) = \GL(n,2) \times
\GL(1,3)$. The relevant subgroups of $\Aut(A)$ are the solvable 
subgroups $U$ of $\Aut(A)$ of order $2^{5-n}$ with $O(U) = \{1\}$. 
Note that $\GL(4,2)$ has two conjugacy classes of subgroups $V$ of order 
$2$. Each of these determine $V \times \GL(1,3) \cong C_2^2$ and thus 
yield three conjugacy classes of subgroups $U$ of order $2$ in $\Aut(A)$.
Of these, two satisfy $O(U) = \{1\}$. The following table lists the 
possibilities for $U$.

\begin{center}
\begin{tabular}{c|c|c|c}
$n$ & $U$ & \# of classes & extensions \\
\hline
4 & $C_2$ & 4 & (96, 159), (96, 160), (96, 218), (96, 219), (96, 230) \\
5 & $\{1\}$ & 1 & (96, 231)
\end{tabular}
\end{center}

%%%%%%%%%%%%%%%%%%%%%%%%%%%%%%%%%%%%%%%%%%%%%%%%%%%%%%%%%%%%%%%%%%%%%%%%%%%%%
\section{Descendants}
\label{fdesc}

In this section we describe an effective method for Algorithm II. Thus 
given a finite group $G$ of $F$-class $c \geq 1$, our aim is to determine
the descendants of $G$ up to isomorphism. Again, we are particularly 
interested in the computation of descendants of finite solvable groups.

We give a brief overview on our method: Its first step is the construction 
of a certain covering group $G^*$ of $G$, see Subsection \ref{sec_cover}. 
This will have the property that every descendant of $G$ is isomorphic to 
a quotient of $G^*$. Further, it will turn out that the isomorphism problem
for descendants of $G$ translates to an orbit-problem for the action of a 
certain group of automorphisms on certain subgroups of $G^*$. Hence the 
covering group construction yields an effective solution of the isomorphism 
problem for descendants. 

In the following subsections we discuss the details of this method. In
the final subsection \ref{algoII} we summarize our resulting method for
Algorithm II.

\subsection{Covering groups of finite groups}
\label{sec_cover}

Let $G$ be a finite group of $F$-rank $\ell = p_1^{d_1} \cdots p_r^{d_r}$ 
and let $k = p_1 \cdots p_r$ be the core of $\ell$. Let $g = \{g_1, \ldots, 
g_n\}$ be an arbitrary generating set of $G$ and let $F$ be free on $n$ 
generators $\{f_1,\ldots, f_n\}$. Let $\mu : F \to G$ be the homomorphism 
induced by mapping $f_i$ to $g_i$ for $1\leq i \leq n$. Denote $R = 
\ker(\mu)$ and let $L$ be the full preimage of $F(G)$ under $\mu$. Then 
we define
\[ K(G) := [R,L]R^k \mbox{ and } 
   G^* := F/ K(G) \text{ and } 
   M(G) := R/ K(G). \]
We call $G^*$ a \Defn{covering group} of $G$ with \Defn{multiplicator} 
$M(G)$ and \Defn{covering kernel} $K(G)$ and we denote $\mu$ as the 
\Defn{presentation epimorphism} of $G$ associated with $G^*$. Further, 
for $p \in \{p_1, \ldots, p_r\}$ we define
\[ G_p^* := F/[R,L]R^p \text{ and } M_p(G) := R/[R,L]R^p. \]
We call $G_p^*$ a \Defn{$p$-covering group} of $G$ with 
\Defn{$p$-multiplicator} $M_p(G)$. The group $R/R' R^p$ is an elementary 
abelian $p$-group of rank $1+(n-1)|G|$ by Schreier's theorem. The 
$p$-multiplicator $M_p(G)$ is isomorphic to $K_p / [K_p, L]$. 

By construction, $M(G) = M_{p_1}(G) \times \ldots \times M_{p_r}(G)$. This
implies that the covering group $G^*$ is a subdirect product of the 
$p$-covering groups $G_{p_1}^*, \ldots, G_{p_r}^*$ with amalgamated factor 
group $G$.

\begin{theorem}
The isomorphism type of $G_p^*$ depends only on $G$, the number of generators
$n$ and the prime $p$, but not on the presentation $\mu$. Similarly, the 
isomorphism type of $G^*$ depends only on $G$ and $n$, but not on $\mu$.
\end{theorem}

\begin{proof}
Let $p \in \{p_1, \ldots, p_r\}$. By \cite[Satz 1]{Gas54} we obtain that
the isomorphism type of the group $F/R'R^p$ depends on $G$ and on $n$, but 
not on $\mu$. As $G_p^* = F/[R,L] R^p \cong (F/R'R^p)/([R,L]R^p/R'R^p)$, 
the result for the $p$-covering group follows. Using that $G^*$ is a 
subdirect product of $G_{p_1}^*, \ldots, G_{p_r}^*$, we now obtain the
same result for the covering group $G^*$.
\end{proof}

\subsection{Covering groups of finite solvable groups}

If $G$ is solvable, then $G^*$ is solvable. Thus $G^*$ can be described
by a polycyclic presentation in this case and this type of presentation
allows effective computations with $G^*$. In this section we describe
how a polycyclic presentation of $G^*$ can be determined. 

Let $\mu : F \to G : f_i \mapsto g_i$ be the presentation epimorphism 
associated with $G^*$ with $R = \ker(\mu)$. Then $R/R'$ is free 
abelian on $(n-1)|G|+1$ generators and $G$ acts by conjugation on 
$R/R'$. We consider each prime $p$ dividing the $F$-rank $\ell$ in 
turn and use the Gaussian elimination algorithm to determine a basis 
for $[R,L]R^p/R'$. This allows us to read off a basis for the quotient 
$(R/R') / ([R,L]R^p/R') \cong R/[R,L]R^p \cong M_p(G)$. This information 
allows us to extend a polycyclic presentation of $G$ to a polycyclic 
presentation of $G_p^*$. In turn, this allows us to determine a polycyclic 
presentation for $G^*$ as subdirect product of $G_{p_1}^*, \ldots, G_{p_r}^*$.  

We usually use a minimal generating set for $G$ to determine $G^*$. We
note that this can be determined effectively as for example described in 
\cite{LMe94}. 

\subsection{Automorphisms of covering groups}
\label{sec_autos}

Let $G$ be a finite group with covering group $G^*$ and multiplicator 
$M = M(G)$. Let $\autGM$ denote the group of automorphisms of $G^*$ which 
leaves $M$ setwise invariant. Every $\alpha \in \autGM$ induces an 
automorphism $\ol{\alpha}$ of $G^*/M$ via $\ol{\alpha}(gM) = \alpha(g)M$. 
We identify $G$ with $G^*/M$ and thus $\ol{\alpha}$ with an automorphism 
of $G$ and obtain the following homomorphisms:
\begin{align*}
  \eta_G &: \autGM \to \Aut(G) \;: \alpha \mapsto \ol{\alpha}, \\
  \eta_M &: \autGM \to \Aut(M) : \alpha \mapsto \alpha|_M .
\end{align*}

We use these homomorphisms to investigate the structure of $\autGM$.
Let $W = \ker(\eta_G)$ and $V = \ker(\eta_G) \cap \ker(\eta_M)$. Then $W$ 
contains those automorphisms of $G^*$ that induce the identity map on the 
quotient $G^*/M \cong G$ and $V$ contains those automorphisms that induce 
the identity map on $G^*/M$ and on $M$. The group $\autGM$ has the 
normal series 
\[ \autGM \unrhd W \unrhd V \unrhd \{id\}. \]

Let $\gamma: G \to \Aut(M) : g \mapsto \ol{g}$ denote the conjugation 
action of $G$ on $M$ in $G^*$ and let $C$ be the centralizer in $\Aut(M)$ of 
the image of $\gamma$. The natural action of $C$ on $M$ induces an 
action of $C$ on $Z^2(G,M)$ via $c(\delta)(g,h) = c(\delta(g,h))$. This 
action leaves $B^2(G,M)$ invariant and hence induces an action on $H^2(G,M)$
via $c([\delta]) = [c(\delta)]$. Let $\epsilon$ denote an element of 
$Z^2(G,M)$ which defines $G^*$ as extension of $M$ by $G$. Recall that
$[\epsilon] = \epsilon + B^2(G,M) \in H^2(G,M)$.

\begin{theorem} \label{natural}
Let $G$ be a finite group with covering group $G^*$ and multiplicator $M$.
\begin{enumerate}
\item
$\eta_G$ is surjective and thus $\autGM/W \cong \Aut(G)$.
\item
$W/V \cong \eta_M(W) = \Stab_C([\epsilon])$.
\item
$V \cong Z^1(G,M)$.
\end{enumerate}
\end{theorem}

\begin{proof}
(a) Let $\mu : F \to G : f_i \mapsto g_i$ be the presentation epimorphism
associated with $G^*$ and let $K(G)$ be the corresponding covering kernel
so that $G^* = F/K(G)$. Let $g_i^* = f_i K(G)$ for $1 \leq i \leq n$. 
Then $g_1^*, \ldots, g_n^*$ generates $G^*$ and $\varphi : G^* \to G : 
g_i^* \mapsto g_i$ is a natural epimorphism from $G^*$ to $G$.

Let $\beta \in \Aut(G)$. Then $\{ \beta(g_1), \ldots, \beta(g_n)\}$ is 
another generating set of $G$. By Gaschütz' theorem \cite[Satz 1]{Gas54}, 
there exists a generating set $h_1, \ldots, h_n$ of $G^*$ so that 
$\varphi(h_i) = \beta(g_i)$ for $1 \leq i \leq n$. 

Let $f_i' \in F$ with $f_i' K(G) = h_i$ for $1 \leq i \leq n$ and
define the homomorphism $\rho : F \to F$ via $\rho(f_i) = f_i'$. Then
$\rho$ induces $\beta$ via $\mu$ on $G$. Thus $\rho(R) \subseteq R$
and $\rho(L) \subseteq L$. Hence $\rho(K(G)) \subseteq K(G)$ and $\rho$ 
induces a homomorphism $\alpha: G^* \to G^*$. By construction,
$\alpha(g_i^*) = h_i$ for $1 \leq i \leq n$ and hence $\alpha$ is
surjective. As $G^*$ is finite, this yields that $\alpha \in \Aut(G^*)$.
In turn, this implies that $\alpha$ is a preimage of $\beta$ under $\eta_G$. 

(b)
Clearly $W/V \cong \eta_M(W)$. It remains to show that $\eta_M(W) = 
\Stab_C([\epsilon])$. We use the ideas of \cite{Rob81} for this purpose. 
First, we recall that the group of compatible pairs $\Comp(G,M)$ is 
defined as $\{(\alpha, \beta) \in \Aut(G) \times \Aut(M) \mid
    \ol{g^\alpha} = \beta^{-1} \ol{g} \beta \text{ for } g \in G\}$.
Next, we recall from \cite{Rob81} that $\Comp(G,M)$ acts on $H^2(G,M)$ 
via $\delta(g,h)^{(\alpha,\beta)} = \delta(g^{\alpha^{-1}}, 
h^{\alpha^{-1}})^\beta$ for $\delta \in Z^2(G,M)$, $(\alpha, \beta) \in
\Comp(G,M)$ and $g,h \in G$. Let \[\eta : \autGM \to \Aut(G) \times
\Aut(M) : \alpha \mapsto (\eta_G(\alpha), \eta_M(\alpha)).\] Then \cite{Rob81}
shows that $\im(\eta) = \Stab_{\Comp(G,M)}([\epsilon]) \leq \Comp(G,M)$. 

Let $\alpha \in W$. Then the definition of $W$ implies that $\eta(\alpha) 
= (\id, \beta) \in \Comp(G,M)$ for some $\beta\in\Aut(M)$.
The definitions of $C$ and $\Comp(G,M)$ imply 
that $\beta \in C$. Hence $\eta(W) \leq \{\id\} \times C \leq \Comp(G,M)$. 
Thus $\eta(W) = \Stab_{\{\id\} \times C}([\epsilon])$. It remains to note
that $\eta(W) = \{\id\} \times \eta_M(W)$ and that the action of $C$ on
$H^2(G,M)$ as defined above coincides with the action of $\{\id\} \times C$ 
as subgroup of $\Comp(G,M)$. This yields the desired result.

(c)
Is folklore (see e.g. \cite[Exercises 11.4]{Rob82}).
\end{proof}

Theorem \ref{natural} and its proof are constructive and indicate a method
to determine $\autGM$ explicitly. First, $Z^1(G,M)$ and thus $V$ can be
computed readily as described in \cite[Section 7.6.1]{HEO05}. Then, a preimage for a 
generator of $\Aut(G)$ under $\eta_G$ can be obtained as described in the
proof of Theorem \ref{natural}. It remains to determine generators for the
image $\eta_M(W)$ together with preimages under $\eta_M$ for each 
generator. We discuss this in more detail.

Recall that $M = M_{p_1} \times \cdots \times M_{p_r}$, where $M_{p_i}$
is elementary abelian of rank $e_i$, say. Thus 
\[ \Aut(M) = \GL(e_1, p_1) \times \cdots \times \GL(e_r, p_r). \]
Further, the group $G^*$ is the subdirect product of $G_{p_1}^*, \ldots,
G_{p_r}^*$ and each $G_{p_i}^*$ is an extension of $M_{p_i}$ with $G$. 
Let $C_{p_i}$ denote the centralizer of the conjugation action of $G_{p_i}$
on $M_{p_i}$. Next, note that
\[ Z^2(G, M) = Z^2(G, M_{p_1}) \times \cdots \times Z^2(G, M_{p_r}).\]
If $\epsilon \in Z^2(G,M)$ defines $G^*$ as extension of $M$ by $G$ and
$\epsilon = \epsilon_{p_1} + \ldots + \epsilon_{p_r}$ with $\epsilon_{p_i}
\in Z^2(G, M_{p_i})$, then $\epsilon_{p_i}$ defines $G_{p_i}^*$ as 
extension of $M_{p_i}$ by $G$. Let $[\epsilon_{p_i}] = \epsilon_{p_i}
+ B^2(G, M_{p_i})$ for $1 \leq i \leq r$. 

The following lemma reduces the determination of $\Stab_C([\epsilon])$ 
to the elementary abelian direct factors of $M$.

\begin{lemma} \label{imageetaM}
Let $W_{p_i} = \eta_M(W) \cap \GL(e_i, p_i)$ for $1 \leq i \leq r$. 
\begin{enumerate}
\item
$\eta_M(W) = W_{p_1} \times \cdots \times W_{p_r}$.
\item
$W_{p_i} = \Stab_{C_{p_i}}([\epsilon_i])$ for $1 \leq i \leq r$.
\end{enumerate}
\end{lemma}

\begin{proof}
Let $C$ be the centralizer in $\Aut(M)$ of the action of $G^*$ on $M$.
The definition of a centralizer induces that $C$ is the direct product
of the groups $C_{p_1}, \ldots, C_{p_r}$. This yields that $\eta_M(W)$ 
is the direct product of the stabilizers $\Stab_{C_{p_i}}([\epsilon_{p_i}])$ 
as required.
\end{proof}

Let $p \in \{p_1, \ldots, p_r\}$ and let $e$ be the rank of $M_p$. The 
centralizer $C_p$ of the action of $G^*$ on $M_p$ can be determined by
the following steps:
\begin{itemize}
\item
Determine the subalgebra $L$ of $\F_p^{e \times e}$ centralizing the
action of $G^*$ on $M_p$.
\item
Determine the unit group $U(L)$ using the radical series of $L$ and
Wedderburn's theorem for each quotient of this series as described in
\cite{Sch00}.
\end{itemize}

Given generators for $C_p$ one can then construct generators for the
stabilizer $\Stab_{C_p}([\epsilon_i])$. The stabilizer is usually
obtained by concurrently constructing the orbit of $[\epsilon_i]$.
This is time-consuming if the orbit is large. Thus is it useful
to reduce $C_p$ to a subgroup containing the stabilizer {\em a priori}.
The following lemma exhibits a method for this purpose. 

\begin{lemma}
\label{trivact}
Let $p \in \{p_1, \ldots, p_r\}$ and let $S$ be a Sylow $p$-subgroup of 
$G^*$. Then $S$ centralizes a series $M_p = M_{p,1} > \ldots > M_{p,s}
= \{1\}$ through $M_p$. Let $T_p = S_p' S_p^p$ and $T_{p,i} = 
(T_p \cap M_{p,i})M_{p,i+1}/M_{p,i+1}$. Then $W_p$ acts trivially
on $T_{p,i}$ for $1 \leq i \leq s$.
\end{lemma}

\begin{proof}
Each element $w \in W_{p_i}$ defines an automorphism $\beta$ of $G_p^*$ 
that centralizes $G \cong G_p^*/M_p$. Note that $M_p \leq S_p$ by 
construction and $\beta$ centralizes $S_p/M_p$. Let $a,b \in S_p$. 
Then $\beta(a) = a m$ and $\beta(b) = b n$ for certain $m,n \in M_p$. 
Thus $\beta([a,b]) = [\beta(a), \beta(b)] = [am, bn] = [a,b]$, as 
$M_p$ is central in $S_p$. Hence $\beta$ induces the identity on 
$S_p'$. Similarly, $\beta(a^p) = \beta(a)^p = (am)^p = a^p$ and thus 
$\beta$ induces the identity on $S_p^p$. This yields the desired result.
\end{proof} 

\subsection{Allowable subgroups}
\label{sec_allow}

Let $G$ be a finite group of $F$-class $c$ with covering group $G^*$ and
multiplicator $M$. Then $G^*$ is a finite group of $F$-class $c$ or $c+1$.
We define $N := \nu_c(G^*)$ and call it the \Defn{nucleus} of $G$. An  
\Defn{allowable  subgroup} $U$ of $G^*$ is a proper subgroup of $M$ which
is normal in $G^*$ and satisfies $M=NU$.

\begin{lemma}
If $U$ is an allowable subgroup of $G^*$, then $G^*/U$ is a descendant of $G$.
\end{lemma}

\begin{proof}
As $U$ is a supplement to $N$ in $M$, if follows that $\nu_c(G^*/U)
= NU/U = M/U$ and hence the result follows.
\end{proof}

The following theorem shows that the allowable subgroups are sufficient
to describe all descendants of $G$.

\begin{theorem}
Let $H$ be an arbitrary descendant of $G$. Then $H \cong G^*/U$ for an
allowable subgroup $U$.
\end{theorem}

\begin{proof}
Let $\mu : F \to G : f_i \mapsto g_i$ be the presentation epimorphism 
associated with $G^*$. Let $R = \ker(\mu)$ and let $K(G)$ be the
corresponding covering kernel so that $G^* = F/K(G)$. 

As $H$ is a descendant of $G$, there exists an epimorphism $\kappa : H \to G$
with kernel $\nu_c(H)$. Let $e=\{e_1,\ldots,e_n\}$ be a subset of $H$ such 
that $\kappa(e_i)=g_i$ for $1\leq i\leq n$. Then $e$ generates $H$, since by
Lemma~\ref{fitting} (a) we have $\nu_c(H) \leq \Frat(H)$ for $c > 0$.
Let $\lambda : F \to H$ the be homomorphism with $\lambda(f_i) = e_i$ for 
$1 \leq i \leq n$. Then $\mu = \kappa \circ \lambda$ by construction.

As $\kappa$ has kernel $\nu_c(H)$, it maps $F(H)$ onto $F(G)=F(H/\nu_c(H))
=F(H)/\nu_c(H)$ by Lemma~\ref{fitting}. It follows that $\lambda$ maps 
$L$ onto $F(H)$ and $R$ onto $\nu_c(H)$. Therefore $K(G)$ is mapped to 
$\nu_{c+1}(H)=\{1\}$ and thus $K(G)$ is contained in the kernel of 
$\lambda$. As $G^* = F/K(G)$, we obtain that $\lambda$ induces a map 
(which we also denote by $\lambda$) from $G^*$ onto $H$. 

Let $U$ be the kernel of $\lambda$. It remains to show that $U$ is a 
allowable subgroup of $G^*$. First, $U$ is a normal subgroup of $G^*$.
Next, $U$ is contained in $M$, as $\lambda \circ \kappa$ is an epimorphism
on $G$ with kernel $M$. Then $U$ is a proper subgroup of $M$, as $H$
has larger $F$-class than $G$. Finally, $U N = M$, as $\lambda(N) = 
\nu_c(H) = \lambda(M)$.
\end{proof}

For a subgroup $U$ of $M$ let $U_p = U \cap M_p$ and note that $U_p$
is the Sylow $p$-subgroup of $U$. If $N$ is the nucleus of $M$ and $U$
is an allowable subgroup, then $U_p N_p = M_p$ and hence $U_p$ 
is a supplement to $N_p$ in $M_p$. Conversely, each allowable subgroup 
can be described as direct product of its Sylow subgroups and hence can 
be composed from supplements to $N_p$ in $M_p$.

\subsection{The isomorphism problem}
\label{sec_isom}

Again, let $G$ be a finite group with covering group $G^*$, multiplicator
$M$ and nucleus $N$. We consider the isomorphism problem for descendants. 
For this purpose we consider the action of $\autGM$ on the set of all
allowable subgroups. We first show that this action is well-defined
and then we prove that it solves the isomorphism problem for descendants.

\begin{lemma}
Let $U$ be an allowable subgroup of $G^*$ and let $\alpha \in \autGM$.
Then $U^{\alpha}$ is an allowable subgroup of $G^*$.
\end{lemma}

\begin{proof}
If $U$ is normal in $G^*$, then $U^\alpha$ is normal as well. Next, the
nucleus $N$ is characteristic in $G^*$ and hence invariant under $\alpha$. 
Therefore $N U^\alpha = (NU)^\alpha = M^\alpha = M$ and thus $U^\alpha$ 
is a supplement to $N$ in $M$.
\end{proof}

\begin{theorem} \label{isomsol}
Let $U_1, U_2$ be two allowable subgroups of $G^*$. Then $G^*/U_1 \cong
G^*/U_2$ if and only if there exists $\alpha \in \autGM$ which maps
$U_1$ onto $U_2$.
\end{theorem}

\begin{proof}
If an $\alpha \in \autGM$ exists with $U_1^\alpha=U_2$, then it induces the 
isomorphism $G^*/U_1 \to G^*/U_2 : gU_1 \mapsto \alpha(g)U_2$. It remains to 
prove the converse. Let $\mu : F \to G : f_i \mapsto g_i$ be the presentation
epimorphism associated with $G^*$. Let $R = \ker(\mu)$ and let $K(G)$ be
the covering kernel so that $G^* = F/K(G)$. Let $\rho : F \to G^*$ be the
natural homomorphism onto the factor group $G^*$.

Let $\beta : G^*/U_1 \to G^*/U_2$ be an isomorphism. For $i=1,2$, let $V_i 
\leq F$ be the full preimage of $U_i$ under $\rho$. Then $G^*/U_i$ is 
naturally isomorphic to $F/V_i$ for $i = 1,2$ and thus $\beta$ induces an 
isomorphism $\hat{\beta} : F/V_1 \to F/V_2$. 

For each $1\leq i \leq n$ choose $y_i \in F$ with $\hat{\beta}(f_iV_1) = 
y_iV_2$. Thus  $G^*/U_2 \cong F/V_2 = \gen{ y_1V_2, \ldots, y_nV_2 }$.
As $U_2= V_2/K$ is finite and $G^*$ is finitely generated, 
we can apply \cite[Satz 1]{Gas55}. Thus there exist $z_1, \ldots, z_n \in F$ 
with $z_i V_2 = y_i V_2$ for $1 \leq i \leq n$ and $G^* = \gen{
z_1K, \ldots, z_nK }$. Now we define the group homomorphism
\[ \sigma : F \to G^* \text{ given by } \sigma(f_i) = z_i K 
\text{ for } 1 \leq i \leq n.\]
Then $\sigma$ is surjective by construction, and for all
$g\in F$ we have $\beta(\rho(g)U_1)=\sigma(g)U_2$. 
Therefore $\sigma(R)U_2 = \beta(\rho(R)U_1)=\beta(M/U_1)=M/U_2$,
since $\beta$ maps $M/U_1=\nu_c(G^*/U_1)$ to $M/U_2=\nu_c(G^*/U_2)$.
Hence $\sigma(R)\leq M=R/K$. Similarly, $\sigma$ maps $L$ into $L/K$.
This yields that $\sigma$ maps $K = [R,L]R^k$ into $K/K=1$,
hence $K\leq \ker\sigma$. Since $G^*=F/K$ is finite,  $\ker(\sigma) = K$.
Thus $\sigma$ induces the automorphism
\[ \alpha: G^* \to G^* \text{ given by } \alpha(f_iK) = z_iK
\text{ for } 1 \leq i \leq n.\]
By construction, $\alpha$ maps $M$ onto $M$ and thus is an element of
$\autGM$. Further, $\alpha$ maps $U_1$ onto $U_2$ and hence is an
automorphism of the desired form.
\end{proof}

Similar to Theorem \ref{isomsol} one can determine the automorphism group
of a descendant. The proof of the following theorem is a variation on the
proof of Theorem \ref{isomsol} and is thus omitted.

\begin{theorem} \label{auto_desc}
Let $H = G^*/U$ be an arbitrary descendant of $G$. Then there exists
the natural epimorphism:  
$\Stab_{\autGM}(U) \to \Aut(H): \alpha \mapsto \alpha_{G^*/U}$.
\end{theorem}

\subsection{Algorithm II for solvable groups}
\label{algoII}

In this section we summarize our method to determine the descendants
for a finite solvable group. The main reason for restricting this
part of the algorithm to finite solvable groups is that this type of
groups can be defined by a polycyclic presentation and this, in
turn, allows effective computations with the considered groups.

{\bf Algorithm} {\it Descendants( $G$ )} \vspace{-0.3cm}
\begin{itemize}
\item
Determine $G^*$ with multiplicator $M$ and nucleus $N$, 
see Section \ref{sec_cover}.
\item
Determine the set $\cL$ of allowable subgroups of $G^*$.
\item
If $|\cL| \leq 1$, then return $\{ G^*/U \mid U \in \cL \}$
and $\{ \Aut(G^*/U) \mid U \in \cL \}$.
\item
Determine $\autGM$ from $\Aut(G)$, see Section \ref{sec_autos}.
\item
Determine orbits and stabilizers for the action of $\autGM$ on $\cL$.
\item
For each orbit representative $U$ determine $H = G^*/U$ and $\Aut(H)$,
see Theorem \ref{auto_desc}.
\item
Return the resulting list of descendants $H$ with their automorphism
groups.
\end{itemize}

In the following, we discuss improvements to the Algorithm {\it Descendants}.
Recall that the multiplicator $M$ is a direct product of elementary abelian 
groups $M = M_{p_1} \times \cdots \times M_{p_r}$. A major improvement can be
obtained by reducing as many computations as possible to each of the direct
factors $M_{p_i}$ and use linear algebra to compute with each such factor.

The factorisation of $M$ induces that the nucleus $N$ and each allowable 
subgroup $U$ splits as well into a similar direct product and $U_{p_i}$ is 
a supplement to $N_{p_i}$ in $M_{p_i}$. Thus if $\cL_{p_i}$ is the set of 
proper supplements to $N_{p_i}$ in $M_{p_i}$, then it follows that 
\[ \cL = \{ U_1 \times \cdots \times U_r 
         \mid U_i \in \cL_{p_i} \text{ for } 1 \leq i \leq r \}.\]

Let $\eta_M : \autGM \to \Aut(M) : \alpha \mapsto \alpha|_M$ the homomorphism 
induced by the action of $\autGM$ on $M$. Then all except the last step of 
Algorithm {\it Descendants} uses as acting group $\Gamma = \im(\eta_M) \leq 
\Aut(M)$ only and not $\autGM$ itself. We note that 
\[ \Aut(M) = \Aut(M_{p_1}) \times \cdots \times \Aut(M_{p_r}), \mbox{ and }\]
\[ \Aut(M_{p_i}) \cong \GL(e_i, p_i) \text{ for } 1 \leq i \leq r \]
and thus $\Gamma$ is a subgroup of this direct product. 
Let $\Gamma_i = \Gamma \cap \Aut(M_{p_i})$ for $1 \leq i \leq r$. Then
\[ \Delta := \Gamma_1 \times \ldots \times \Gamma_r \leq \Gamma.\]

The computation of orbits under the action of the direct product $\Delta$
splits into the computation of the orbits in each direct factor. Hence,
instead of one computation with a long orbit, we obtain $r$ orbit 
computations with shorter orbits. This induces a significant reduction
for Algorithm {\it Descendants}.

However, the examples in Section \ref{examII} show that $\Delta$ can be 
a proper subgroup of $\Gamma$ and $\Gamma$ can act non-trivially on 
$\Delta$-orbits. Thus we cannot reduce the computation of descendants
to the Sylow subgroups of $M$ completely. However, we determine and
use the subgroup $\Delta$ of $\Gamma$ to reduce the arising orbit 
computations. For this purpose it is useful to observe that the image 
of $W$ in $\Aut(M)$ is a subgroup of $\Delta$, see Section \ref{sec_autos}.

\subsection{Examples}
\label{examII}

Let $G_1$ be the symmetric group on four points, $G_2$ the group $(36,3)$
from the small groups library and $G_3$ the group $(72, 22)$. The following
table lists the orders of the multiplicators and nucleuses of these groups
with respect to minimal generating sets of the underlying groups. The 
table then exhibits the descendants of these groups. In the cases of $G_1$ 
and $G_2$ we list the descendants explicitly by their number in the small 
groups library. In the case of $G_3$ we describe them by their stepsizes
$o$, where the stepsize of a descendant $H$ is $o = |H| / |G_3| = |H|/72$.
Note that an entry $o^x$ means that there are $x$ descendants of stepsize $o$.

\begin{center}
\begin{tabular}{llll}
group & $|M|$ & $|N|$ & descendants \\
\hline
$G_1$ & $2^8$ & $2^3$ 
   & (48, 28), (48, 29), (96, 64), (192, 180), (192,181) \\ 
\hline
$G_2$ & $2^5 \cdot 3^3$ & $2^3 \cdot 3$ 
   & (72, 3), (144, 3), (288, 3), (108, 3), (216, 3), (432,3), (864, 3) \\
\hline
$G_3$ & $2^5 \cdot 3^6$ & $2 \cdot 3^3$ 
   & $2^3, 3^2, 6^7, 9^2, 18^7, 27^1, 54^3$ 
\end{tabular}
\end{center}

For $G_2$ and $G_3$ the multiplicators are direct products of their
Sylow 2- and 3-subgroups. We note that in the notation of Section
\ref{algoII}, the group $\Delta$ is a proper subgroup of $\Gamma$ in
both cases. This has an impact on the descendant computation in the
case of $G_3$: there are three descendants of stepsize $2$, two descendants 
of stepsize $3$, but $7 > 2 \cdot 3$ descendants of stepsize $2 \cdot 3$. 
This shows that the orbit-stabilizer computation used in Algorithm
{\it Descendants} does not fully reduce to the Sylow subgroups of the
multiplicator.

%%%%%%%%%%%%%%%%%%%%%%%%%%%%%%%%%%%%%%%%%%%%%%%%%%%%%%%%%%%%%%%%%%%%%%%%%%%%%
\section{Implementation}
\label{impl}

We have implemented our methods to construct finite solvable groups
as GAP Package GroupExt \cite{groupext}. Our implementation uses
a variety of algorithms from GAP and other packages. In particular,
it uses the following.

\begin{itemize}
\item
The FGA package \cite{FGA} for computations in free groups.
\item
The Polycyclic package \cite{polycyc} for computations in polycyclic
groups.
\item
The AutPGroup package \cite{autpgrp} to determine the automorphism 
groups of finite $p$-groups.
\item
The genss package \cite{genss} for computing stabilizers using a
randomised Schreier-Sims algorithm.
\end{itemize}

%%%%%%%%%%%%%%%%%%%%%%%%%%%%%%%%%%%%%%%%%%%%%%%%%%%%%%%%%%%%%%%%%%%%%%%%%%%%%
\section{The groups of a given order}
\label{applics}

Our implementations of Algorithms I and II for solvable groups can be 
used to determine the solvable non-nilpotent groups of a given order $o$. 
By induction, we assume that all solvable non-nilpotent groups of order 
properly dividing $o$ are given. We then proceed in two steps:

\begin{itemize}
\item[(1)]
Determine all solvable non-nilpotent groups of order $o$ and $F$-class $1$ 
up to isomorphism; For this purpose consider every proper divisor $\ell$ of 
$o$ and determine the solvable groups of order $o$, $F$-class $1$ and
$F$-rank $\ell$.
\item[(2)]
For every solvable non-nilpotent group $G$ of order $s$, where $s$ 
is a proper divisor of $o$, determine the descendants of $G$ of order 
$o$.
\end{itemize}

We have used this to determine (again) all solvable non-nilpotent groups
of order at most $2000$ as available in the small groups library. This
provides the first independent check for the correctness of the Small
Groups library.

\subsection{The groups of order 2304}

We have used the algorithm described in this paper to determine (for the 
first time) the groups
of order $2304 = 2^8 \cdot 3^2$. We provide some summary information on
these groups in this section. The groups themselves will be made
available as part of a forthcoming GAP package \cite{smallext}.

The nilpotent groups of order $2304$ can be obtained readily as direct
products of $p$-groups and thus they can be considered as availble; There 
are $112184$ of them. In the following we concentrate on the non-nilpotent 
groups of order $2304$. Every such group is solvable by Burnside's 
$pq$-Theorem.

We first consider the non-nilpotent groups of $F$-class $1$. Table
\ref{table:fclass1} lists their possible $F$-ranks $\ell$ and for each
$\ell$ the number of solvable groups of order $2304$, $F$-class $1$
and $F$-rank $\ell$. We omit those divisors $\ell$ of $o$ which do not
lead to any groups. In summary, there are 1953 non-nilpotent groups of
order $2304$ of $F$-class $1$. Note that $C_2^8\times C_3^2$ has order
$2304$, $F$-class $1$ and $F$-rank $2304$ but is abelian and hence not
included.

\begin{center}
\(
\begin{array}{D{=}{\;=\;}{5}|r}
\multicolumn{1}{c}{$F$\text{-central rank}} & \text{\# of groups} \\
\hline
32 = 2^5           & 8 \\
64 = 2^6           & 37 \\
128 = 2^7          & 28 \\
144 = 2^4\cdot 3^2 & 193 \\
192 = 2^6\cdot 3   & 208 \\
256 = 2^8 & 9 \\
288 = 2^5\cdot 3^2 & 834 \\
384 = 2^7\cdot 3 & 54 \\
576 = 2^6\cdot 3^2 & 558 \\
768 = 2^8\cdot 3 & 8 \\
1134 = 2^7\cdot 3^2 & 16 \\
\end{array}
\)
\captionof{table}{Groups of F-central class 1}
\label{table:fclass1}
\end{center}

Next we consider the groups of $F$-class $c > 1$. To determine these, 
we have to consider each group of order properly dividing $2304$ and 
determine descendants of order $2304$. As a preliminary step, it is
useful to determine {\it a priori} which groups of a given divisor
$s$ of $o$ could have descendants of order $o$. The following lemma
is highly useful for this purpose.

\begin{lemma}
 Let $G$ be a non-nilpotent group of order $2^a \cdot 3$ such that $G$ has
 descendants of order $2^a \cdot 9$. Then $F(G)\cong C_2^{a-1}\times C_3$,
 and $G/O_2(G)\cong S_3$.
\end{lemma}

\begin{proof}
 Since $G$ has a descendant of order $3\cdot \abs{G}$, its $F$-central
 rank is divisible by $3$, hence $\abs{F(G)}$ is divisible by $3$. But then
 $F(G)$ contains a $3$-Sylow subgroup $P_3\cong C_3$ of $G$. Since
 $F(G)$ is nilpotent, it splits into a direct product of $P_3$ and a
 $2$-group. Since $F(G)$ is normal in $G$, also $P_3$ is normal in $G$.
 Hence $O_3(G)=P_3\cong C_3$.

 By Sylow's theorem, the number of $2$-Sylow subgroups in $G$ is
 one or three. If there was only one, then it would be normal and $G$
 would be nilpotent. Hence there are three and the conjugation action of
 $G$ on the set of its $2$-Sylow subgroups yields a homomorphism into
 $S_3$ whose kernel is a normal $2$-subgroup of $G$. Since the kernel is
 not a full 2-Sylow subgroup (else $G$ would be nilpotent),
 it must have index 6 in $G$ and thus order $2^{a-1}$. Clearly, this
 kernel coincides with $O_2(G)$. Therefore $G/O_2(G)\cong S_3$.

 It remains to show that $O_2(G)\cong C_2^{a-1}$. For this, suppose that $H$
 is a descendant of $G$ of order $2^a \cdot 9$ and $F$-central class $c$. Thus
 $H/\nu_c(H)\cong G$, hence $\nu_c(H)\cong C_3$. Since the $F$-central
 class of $H$ is larger than that of $G$, the only way this is possible
 is if $c=2$ and the preimage of $P_3$ in $H$ is isomorphic to $C_9$.
 But then $G$ has $F$-central class 1 and $F(G)$ is a direct product of
 elementary abelian subgroups.
\end{proof}
We can apply this to the $1\,090\,235$ groups of order $768=2^7\cdot 3$ to
determine those which may have descendants of order $2304$. It turns out
that only 8 groups remain, each having precisely one descendant of the
order $2304$.

A summary of the descendant computation is given in Table 
\ref{table:fdesc}.

\begin{center}
\begin{tabular}{r|r|r|r|r}
order & \# groups & \# non-nilpotent
    & \# grps w/ descendants & \# descendants \\
\hline
  6 &      2 & 1 & 0 & 0\\
 12 &      5 & 3 & 0 & 0\\
 18 &      5 & 3 & 0 & 0\\
 24 &     15 & 10 & 0 & 0\\
 36 &     14 & 10 & 0 & 0\\
 48 &     52 & 38 & 4 & 34\,210\\
 72 &     50 & 40 & 2 & 6\\
 96 &    231 & 180 & 5 & 728\,926\\
144 &    197 & 169 & 21 & 68\,945\\
192 &  1\,543 & 1\,276 & 6 & 24\,889\\
288 &  1\,045 & 943 & 116 & 10\,835\,672 \\
384 & 20\,169 & 17\,841 & 7 & 426 \\
576 &  8\,681 & 8\,147 & 865 & 1\,980\,937\\
768 & 1\,090\,235 & 1\,034\,143 & 8 & 8\\
1152 &  157\,877 & 153\,221 & 47\,848 & 1\,967\,974\\
\hline
& 1\,280\,121 & 1\,216\,025 & 48\,882 & 15\,641\,993\\
\end{tabular}
\captionof{table}{Groups of order $2304$ and $F$-class $c>1$.}
\label{table:fdesc}
\end{center}

In total there are $112\,184+1\,953+15\,641\,993=15\,756\,130$ isomorphism 
types of groups of order 2304. In comparison, there are $10\,494\,213$ groups 
of order $512$.

Table \ref{table:top-ten} contains the top-ten among the groups of order
properly dividing $o$; that is, those groups with the ten largest
numbers of descendants. It is interesting to note that over 56\% of
the groups of order 2304 are descendants of a single group.

\begin{center}
\[
\begin{array}{D{,}{,\;}{5}|r}
\multicolumn{1}{c}{\text{group}} & \text{descendants} \\
\hline 
(288,1040) & 8\,937\,790\\
(576,8590) & 707\,578 \\
(96,230) & 696\,554 \\
(288,1043) & 696\,554 \\
(288,1044) & 696\,554 \\
(576,8588) & 203\,006 \\
(288,976) & 160\,928 \\
(576,8582) & 131\,664 \\
(576,8675) & 120\,310 \\
(576,8589) & 110\,292 \\
\end{array}
\]
\captionof{table}{Top ten groups in terms of number of descendants.}
\label{table:top-ten}
\end{center}

\bibliographystyle{abbrv}
\bibliography{/home/beick/tex/bibdata/all.bib}
\end{document}